\documentclass[11pt]{amsart}
\usepackage{amsmath, amsthm, amssymb, amsfonts,color,url,enumitem, caption, booktabs}
 \usepackage{subcaption,comment}
\usepackage{float}  

\usepackage{cancel}
\usepackage{wrapfig,hyperref}
\usepackage{mathtools}
\usepackage{fullpage}
\usepackage{color}
\usepackage[normalem]{ulem}
\usepackage{tikz}
\usetikzlibrary{decorations}
\usetikzlibrary{decorations.text}
\usetikzlibrary{decorations.pathreplacing}
\theoremstyle{definition}
\newtheorem{example}{Example}[section]
\newtheorem{question}{Question}[section]
\newtheorem{definition}[example]{Definition}
\newtheorem{theorem}[example]{Theorem}
\newtheorem{corollary}[example]{Corollary}
\newtheorem{lemma}[example]{Lemma}
\newtheorem{proposition}[example]{Proposition}
\newtheorem{remark}[example]{Remark}

\newtheorem{innercustomthm}{Theorem}
\newenvironment{customthm}[1]
  {\renewcommand\theinnercustomthm{#1}\innercustomthm}
  {\endinnercustomthm}

\newcommand\dyckpath[5]{
  \begin{scope}[local bounding box=#4]
    \fill[white]  (#1) rectangle +(#2,#2);
    \fill[red!25!white] (#1) foreach \dir in {#3}{-- ++(\dir*90:1)} |- (#1);
    \path[fill] (#1) foreach \i [count=\j] in {0,...,#5}{ +(\i,0) node[anchor=north]{\j} \ifnum\i>#2 circle (1pt) \fi};
    \draw[help lines] (#1) grid +(#2,#2);
    \draw[line width=2pt] (#1) foreach \dir in {#3}{ -- ++(\dir*90:1)};
  \end{scope}
}

\newcommand{\lcm}{\mathrm{lcm}}
\newcommand{\arith}[1]{\textbf{Arith(}#1\textbf{)}}
\newcommand{\sarith}[1]{\textbf{SArith(}#1\textbf{)}}
\newcommand{\diag}{\text{diag}}

\newcommand{\N}{\mathbb{N}}

\newcommand{\C}{\mathbb{C}}
\newcommand{\CT}[2]{\mathrm{CT}(#1,#2)}

\title{Arithmetical Structures on Coconut Trees}

\author{Alexander Diaz-Lopez}\thanks{A. Diaz-Lopez was supported by NSF DMS-2211379.}
\author{Brian Ha}
\author{Pamela E. Harris}\thanks{P.~E.~Harris was supported through a Karen Uhlenbeck EDGE Fellowship.}
\author{Jonathan Rogers}
\author{Theo Koss}
\author{Dorian Smith}
\address[A.~Diaz-Lopez]{Villanova University, Department of Mathematics and Statistics, Villanova, PA 19085}
\email{\textcolor{blue}{\href{mailto:alexander.diaz-lopez@villanova.edu}{alexander.diaz-lopez@villanova.edu}}}

\address[B.~Ha]{University of Colorado, Boulder, Department of Applied Mathematics, Boulder, Colorado, 80309}
\email[B.~Ha]{\textcolor{blue}{\href{mailto:brian.ha@colorado.edu}{brian.ha@colorado.edu}}}

\address[P.~E.~Harris, T.~Koss]{University of Wisconsin, Milwaukee
Department of Mathematical Sciences,
Milwaukee, WI 53227}
\email[P.~E.~Harris, T.~Koss]{\textcolor{blue}{\href{mailto:peharris@uwm.edu}{peharris@uwm.edu}}, \textcolor{blue}{\href{mailto:takoss@uwm.edu}{takoss@uwm.edu}}}

\address[J.~Rogers]{Williams College,
Department of Mathematics and Statistics,
Williamstown, MA 01267}
\email[J.~Rogers]{\textcolor{blue}{\href{mailto:jsr7@williams.edu}{jsr7@williams.edu}}}

\address[D.~Smith]{University of Minnesota Twin Cities,
Department of Mathematics and Statistics,
Minneapolis, MN 55455}

\email[D.~Smith]{\textcolor{blue}{\href{mailto:smi01055@umn.edu}{smi01055@umn.edu}}}

\date{\today}

\begin{document}

\begin{abstract}
If G is a finite connected graph, then an arithmetical structure on $G$ is a pair of vectors $(\textbf{d}, \textbf{r})$ with positive integer entries such that $(\diag(\textbf{d}) - A)\cdot \textbf{r} = \textbf{0}$, where $A$ is the adjacency matrix of $G$ and the entries of \textbf{r} have no common factor other than $1$. In this paper, we generalize the result of Archer, Bishop, Diaz-Lopez, Garc\'ia Puente, Glass, and Louwsma on enumerating arithmetical structures on bidents (also called coconut tree graphs $\CT{p}{2}$) to all coconut tree graphs $\CT{p}{s}$ which consists of a path on $p>0$ vertices to which we append $s>0$ leaves to the right most vertex on the path. We also give a characterization of smooth arithmetical structures on coconut trees when given number assignments to the leaf nodes.
\end{abstract}

\maketitle
\section{Introduction}

An arithmetical structure on a graph is an assignment of weights to the vertices of the graph with positive integers satisfying: the weight at a vertex divides the sum of its neighbors' weights, and the weights do not share any common factor other than $1$. Arithmetical structures arose out of the study of the intersection of degenerating curves in algebraic geometry. Lorenzini  established that finite simple graphs have finitely many arithmetical structures~\cite[Lemma~1.6]{lorenzini_1989}. 

In light of Lorenzini's result, it is of interest to enumerate arithmetical structures on families of graphs. Throughout, we let $\arith{G}$  denote the set of arithmetical structures on a graph $G$ and $|\arith{G}|$ denotes its cardinality.
In \cite{braun_2018}, Braun, Corrales, Corry, Garc\'ia Puente, Glass, Kaplan, Martin, Musiker, and Valencia established if $\mathcal P_{n+1}$ is the path graph on $n+1$ vertices, then $|\arith{\mathcal P_{n+1}}|=C_{n}=\frac{1}{n+1}\binom{2n}{n}$, the $n$-th Catalan number \cite[\href{https://oeis.org/A000108}{A000108}]{OEIS}, 
and if $\mathcal{C}_n$ is the cycle graph on $n$ vertices, then $|\arith{\mathcal{C}_n}| =\binom{2n - 1}{n-1}$ $ = (2n-1)C_{n-1}$. Some partial enumeration results are known for bidents~\cite{archer_2020}, paths with a double edge~\cite{glass20}, and $E_n$ graphs~\cite{vetter21}. For complete graphs, arithmetical structures are in bijection with Egyptian fractions summing to 1 \cite[A002967]{OEIS}. In \cite{keyes21}, Keyes and Reiter provide a (very large) upper bound on the number of arithmetical structures on a graph based on the number of edges of the graph.

\begin{figure}[ht]
    \centering
    \begin{tabular}{cc}
\begin{subfigure}[b]{.35\textwidth}
\begin{tikzpicture}
    \node[]at (2.52,0){$\cdots$};
    \tikzstyle{every node}=[draw,circle,fill=white,minimum size=4pt,inner sep=2pt]
    \node[](v1) at (0,0){};
    \node[](v2) at (1,0){};
    \node[](v3) at (2,0){};
    \node[](v4) at (3,0){};
    \node[](v5) at (4,0){};
    \node[](v6) at (5,.75){};
    \node[](v7) at (5,-.75){};
    \draw(v1)--(v2)--(v3);
    \draw (v4)--(v5)--(v6);
    \draw(v5)--(v7);
        \def\C{(5.75,1)} 
     \def\D{(5.75,-1)} 
     \def\E{(0,-.75)} 
     \def\F{(4,-.75)} 
\draw[decorate,decoration={brace,mirror}] \E -- \F;

  \tikzstyle{every node}=[]
    
   \coordinate (B) at (2,-1);
 \node[yshift=-.25cm] at (B) {$p$ nodes};
    \end{tikzpicture}
         \caption{$\CT{p}{2}$}
        \label{fig:CTp2}
     \end{subfigure}
    &
    \begin{subfigure}[b]{.3\textwidth}
    \begin{tikzpicture}
    \node at (5.2,.1){$\vdots$};
    \node[]at (2.52,0){$\cdots$};
    \tikzstyle{every node}=[draw,circle,fill=white,minimum size=4pt,inner sep=2pt]
    \node[](v1) at (0,0){};
    \node[](v2) at (1,0){};
    \node[](v3) at (2,0){};
    \node[](v4) at (3,0){};
    \node[](v5) at (4,0){};
    \node[](v6) at (5,.75){};
    \node[](v7) at (5,-.75){};
    \node[](v8) at (5.2,-.5){};
    \node[](v9) at (5.2,.5){};
    
    \draw(v1)--(v2)--(v3);
    
    \draw (v4)--(v5)--(v6);
    \draw(v9)--(v5)--(v7);
    \draw(v5)--(v8);
     \def\C{(5.75,1)} 
     \def\D{(5.75,-1)} 
     \def\E{(0,-.75)} 
     \def\F{(4,-.75)} 
  \draw[decorate,decoration={brace,mirror}] \D -- \C;
\draw[decorate,decoration={brace,mirror}] \E -- \F;

  \tikzstyle{every node}=[]
    
  \coordinate (A) at (6,0);
   \coordinate (B) at (2,-1);
 \draw[] (A) circle (0);
 \node[xshift=.65cm] at (A) {$s$ nodes};
 \node[yshift=-.25cm] at (B) {$p$ nodes};
  \end{tikzpicture}
    \caption{$\CT{p}{s}$}
         \label{fig:CTps}
     \end{subfigure}
    \end{tabular}
    \caption{Coconut tree graphs.}
    \label{fig:coconuttree}
\end{figure}
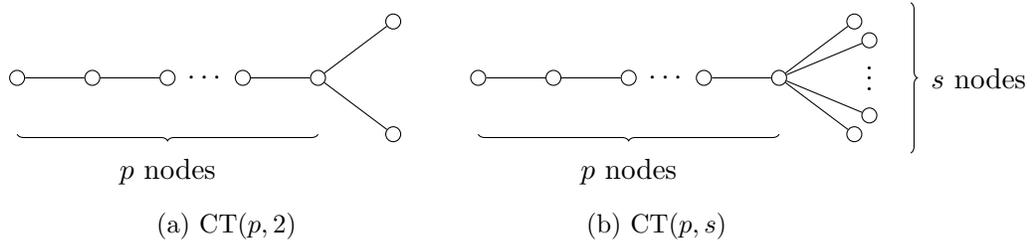

In this work, we enumerate arithmetical structures on the coconut tree graph $\CT{p}{s}$, which consists of a path on $p>0$ vertices to which we append $s>0$ leaves to the rightmost vertex on the path. Figures~\ref{fig:CTp2} and~\ref{fig:CTps} illustrate $\CT{p}{2}$ and $\CT{p}{s}$, respectively. 
We refer to the vertex to which we append leaves as the \textbf{central} vertex of the coconut tree.
The coconut tree graph $\CT{p}{2}$ is also referred to as a bident, and Archer et al. studied arithmetical structures on bidents by counting a smaller subset of arithmetical structures, which they called smooth arithmetical structures \cite{archer_2020}.
In this context, an arithmetical structure is said to be \textbf{smooth} if, for each noncentral vertex, the sum of the weights of its neighbors is $k$ times its weight for some $k$ strictly greater than 1. This condition guarantees that the structure cannot be obtained from a smaller structure by a process called subdivision. The technical definitions of these terms are  given in Definitions~\ref{def:smooth} and \ref{def:subdivision}.
Throughout, we will let $\sarith{G}$ denote the set of smooth arithmetical structures on a graph $G$. 
See Figure~\ref{fig:CT-3-3} for an example of a smooth arithmetical structure on $\CT{3}{3}$.

\begin{figure}
    \centering
    \begin{tikzpicture}
    \tikzstyle{every node}=[draw,circle,fill=white,minimum size=4pt,inner sep=2pt]
    \node[](v1) at (4,1){3};
    \node[](v2) at (4,0){3};
    \node[](v3) at (4,-1){2};
    \node[](v4) at (1,0){2};
    \node[](v5) at (2,0){4};
    \node[](v6) at (3,0){6};
    \draw(v1)--(v6)--(v5)--(v4);
    \draw(v2)--(v6)--(v3);
    \end{tikzpicture}
    \caption{A smooth arithmetical structure on 
    $\CT{3}{3}$.}
    \label{fig:CT-3-3}
\end{figure}
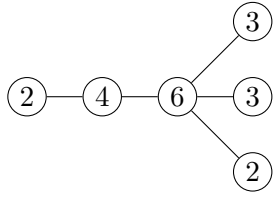

Our main result generalizes the work of Archer et~al.~\cite{archer_2020}, and enumerates arithmetical structures on coconut trees $\CT{p}{s}$, for all $p > 0$ and $s\geq 1$. 
Let $B(n, k) = \frac{n-k+1}{n+1}\binom{n+k}{n}$ be a ballot number 
\cite[\href{https://oeis.org/A009766}{A009766}]{OEIS}, which were introduced by Carlitz \cite{Carlitz:1972}. We can now state our main result, which reduces the problem of counting arithmetical structures on coconut trees
to counting smooth arithmetical structures.

\begin{customthm}{\ref{thm:CT-number-of-AS}}
If $p\geq 1, s\geq 2$ then the number of arithmetical structures for the coconut tree $\CT{p}{s}$ is given by
\begin{align}\label{eq:recursion with A}
|\arith{\CT{p}{s}}|  &=  \sum_{j=0}^s \binom{s}{j} A(p+s - j, s - j)
\end{align}
where
\begin{align*}
    A(x, 0) &= C_{x-1} \mbox{ for $x\geq 1$,}\\
    A(x, 1) &= C_{x-1} - C_{x-2}\mbox{ for $x\geq 2$, and}\\
    A(x, y) &= \sum_{i=y+1}^{x} B(x-y-1, x-i)|\sarith{\CT{i-y}{y}}|\mbox{ for $x\geq 3, y\geq2$.}
\end{align*}
\end{customthm}

This paper is organized as follows. In Section~\ref{sec:background}, we give the necessary background on arithmetical structures on graphs, including the definition of smooth arithmetical structures. In Section \ref{sec:smothsub}, we describe the smoothing and subdividing operations needed in order to count arithmetical structure on coconut tree graphs. In Section~\ref{sec:4}, we prove Theorem~\ref{thm:CT-number-of-AS} counting arithmetical structures on coconut trees
by extending the definition of smooth arithmetical structures to coconut trees with $s\geq 2$. We end with Section~\ref{sec:futurework} by presenting a few directions for further study. 

\begin{remark}
    We point the interested reader to \href{https://github.com/rabidrabbit/arithmetic-structures}{GitHub} containing code whose inputs are a specific graph $G$ and a maximum value $m$, and whose output is the set of arithmetical structures on $G$ with values less than or equal to $m$. 
    \end{remark}

\section{Background on (smooth) arithmetical structures}\label{sec:background}
In this section, we provide the necessary definitions and notations to make our approach precise. Throughout our work $G = (V, E)$ is a simple undirected connected finite graph with vertex set $V$ and edge set $E$. The vertices $u$ and $v$ are \textbf{adjacent} if there exists an edge between them. Denoting the set of vertices of $G$ by $V=\{v_1,v_2,\ldots, v_n\}$, the \textbf{adjacency matrix} $A=(a_{i,j})$ of $G$ is the square-matrix of dimension $|V| \times |V|$ defined by
\[a_{i,j} = \begin{cases}
    1&\mbox{if $v_i$ and $v_j$ are adjacent}\\
    0&\mbox{otherwise}.
\end{cases}\]
Observe that $A$ is symmetric. 

\begin{definition}\label{def:astructure}
Let $G$ be a graph. An \textbf{arithmetical structure} on $G$ is defined as a pair of vectors $(\textbf{d}, \textbf{r})\in \mathbb{N}^{|V|}\times \mathbb{N}^{|V|}$ that satisfy
\begin{equation}\label{eq:defarith}
(\text{diag}(\textbf{d}) - A) \cdot \textbf{r}  =  {\textbf{0}},
\end{equation}
 the entries of $\textbf{r}$ have no common factors other than $1$, and where $\text{diag}(\textbf{d})$ is the diagonal matrix with entries given by the vector $\textbf{d}$ and $A$ is the adjacency matrix of $G$.
\end{definition}
Equivalently, an arithmetical structure on $G$ can be defined as an assignment of numbers to the vertices of $G$ such that:\begin{enumerate}
    \item The assigned number of a vertex divides the sum of its neighbors' assigned numbers.
    \item The gcd of all the assigned numbers is 1.
\end{enumerate}
Throughout, we refer to most arithmetical structures by their $\textbf{r}=(r_1, r_2, \dots, r_n)$ vector as it represents the assignment of values to the vertices of the graph $G$. Equation \eqref{eq:defarith} means that the entries of the vector $\mathbf{d}$ measure by what factor each vertex divides the sum of its neighbors. Hence, $\textbf{r}$ completely determines $\mathbf{d}$. Conversely, if $\mathbf{d}$ satisfies Equation \eqref{eq:defarith} for some vector $\mathbf{r}$ with positive coefficients, Lorenzini \cite[Proposition 1.1]{lorenzini_1989} showed that the matrix $(\text{diag}(\textbf{d}) - A)$ has rank $|V|-1$, hence its kernel is one-dimensional and there is a unique vector $\mathbf{r}$ such that $(\textbf{d},\mathbf{r})$ is an arithmetical structure (since the entries of $\mathbf{r}$ must be positive and have $1$ as their only common factor). We do remark that just having a $\textbf{d}$ vector that satisfies \eqref{eq:defarith} does not automatically imply that there is a vector in the kernel of $(\text{diag}(\textbf{d}) - A)$ with strictly positive coefficients \cite[Remark 3.11]{CorralesValencia}.

 We now formally define coconut tree graphs.
\begin{definition}
A coconut tree $\CT{p}{s}$ for $p>0$ and $s>0$ is a path graph $\mathcal{P}_p$ with $s$ leaf vertices at one of the ends. Note that
$\CT{p}{1} = \mathcal{P}_{p+1}$ and $\CT{p}{0}=\mathcal{P}_p$.
\end{definition}

Given a coconut tree graph $\CT{p}{s}$, we label the path portion of the graph by $v_1$ to $v_p$, and denote the leaf vertices as $v_{\ell_1}, \dots, v_{\ell_s}$, such that $v_p$ connects to all of the leaf vertices. This is illustrated in Figure~\ref{fig:illustratingnotation1}.
This is the same labeling we use for the vectors $\mathbf{r}$ and $\mathbf{d}$, and we write $\mathbf{r}=(r_1,\dots,r_p,r_{\ell_1},\dots,r_{\ell_s})$ and $\mathbf{d}=(d_1,\dots,d_p,d_{\ell_1},\dots,d_{\ell_s})$. Note that this differs slightly from the notation used in Archer et~al. in \cite{archer_2020}.

\begin{figure}[ht]
    \centering
     \begin{tikzpicture}
    \node[](vd) at (6,.2){$\vdots$};
    \node[](vc) at (4,.5){$\cdots$};
    \tikzstyle{every node}=[draw,circle,minimum size=4pt,inner sep=2pt]
    \node[](v1) at (6,2){$v_{\ell_1}$};
    \node[](v2) at (6,1){$v_{\ell_2}$};
    \node[](v3) at (6,-1){$v_{\ell_s}$};
\node[](v4) at (1,.5){$v_1$};
    \node[](v5) at (2,.5){$v_2$};
    \node[](v6) at (3,.5){$v_3$};
    \node[](v7) at (5,.5){$v_{p}$};
    
    \draw(v1)--(v7);
    \draw(v2)--(v7);
    \draw(v3)--(v7);
    \draw(v4)--(v5)--(v6);
    \draw(vc)--(v6);
    \draw(vc)--(v7);
    
    \end{tikzpicture}
    \caption{The
    coconut tree $\CT{p}{s}$.}
    \label{fig:illustratingnotation1}
\end{figure}
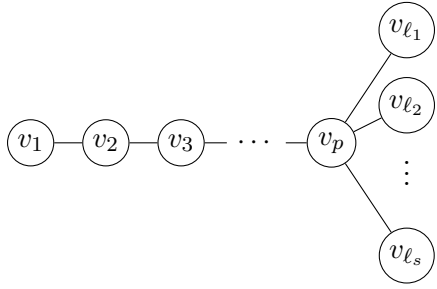

\begin{definition}\label{def:smooth}
An arithmetical structure $(\textbf{d}, \textbf{r})$ on 
$\CT{p}{s}$ is \textbf{smooth} if $$d_1,\dots,d_{p-1},d_{\ell_1}, \dots, d_{\ell_s}\geq 2.$$ 
The set of smooth arithmetical structures on $\CT{p}{s}$ is denoted $\sarith{\CT{p}{s}}$.
\end{definition}

Note that the above definition does not place any restrictions on $d_p$. In the case of 
$\CT{p}{2}$,  Archer et al.~\cite{archer_2020} proved that $d_p$ must be $1$ for smooth arithmetical structures. For $s\geq 3$, this result no longer holds. For example, consider the smooth arithmetical structure on $\CT{3}{3}$ presented in Figure~\ref{fig:CT-3-3}.
 Observe that $\mathbf{r}=(r_1, r_2, r_3,r_{\ell_1}, r_{\ell_2}, r_{\ell_3}) = (2,4,6,3, 3, 2)$ and $\mathbf{d}=(2,2,2,2,2,3)$.

\begin{lemma}\label{lem:CT-increasing-r}
Let $(\textbf{d}, \textbf{r})$ be an arithmetical structure on 
$\CT{p}{s}$. The following conditions are equivalent:
\begin{enumerate}
    \item $d_i \geq 2$ for $i\leq p-1$.
    \item $0 < r_2 - r_1 \leq \dots \leq r_{p-1}-r_{p-2} \leq r_p-r_{p-1}$.
    \item $r_1 < r_2 < \dots < r_{p}$.
\end{enumerate}
\end{lemma}

\begin{proof}
The proof is identical to Lemma 2.1 of \cite{archer_2020} as the proof does not refer to any leaf vertices.
\end{proof}

\begin{lemma}\label{lem:CT-smooth-equivalency}
An arithmetical structure $(\textbf{d}, \textbf{r})$ on $\CT{p}{s}$ is smooth 
 if and only if $r_{\ell_j} < r_p$ for all $1\leq j\leq s$ and $r_1 < r_2 < \dots < r_{p}$. 
\end{lemma}
\begin{proof}
$(\Rightarrow)$ Let $(\mathbf{d},\mathbf{r})$ be a smooth structure on $\CT{p}{s}$. By definition of smooth, $d_i\geq2$ for all $1\leq i\leq p-1$ and $d_{\ell_j}\geq2$ for all $1\leq j\leq s$. 
Since $d_{\ell_j}r_{\ell_j}=r_p$, the condition $d_{\ell_j}\geq2$ implies $r_{\ell_j}<r_p$ for $1\leq j\leq s$, and by Lemma~\ref{lem:CT-increasing-r} item (3), $r_1<r_2<\dots<r_p$.

$(\Leftarrow)$ Let $(\mathbf{d},\mathbf{r})$ be an arithmetical structure on $\CT{p}{s}$, such that $r_{\ell_j}<r_p$ for all $1\leq j\leq s$, and $r_1<r_2<\dots<r_p$. 
By the divisibility condition, $d_{\ell_j}r_{\ell_j}=r_p$ for $1\leq j\leq s$, and since $r_{\ell_j}<r_p$, this gives $d_{\ell_j}\geq2$. By Lemma~\ref{lem:CT-increasing-r}, $r_1<r_2<\dots<r_p$ is equivalent to $d_i\geq2$ for $1\leq i\leq p-1$. Therefore $d_i\geq2$ for all $1\leq i\leq p-1$ and $d_{\ell_j}\geq2$ for all $1\leq j\leq s$. Hence, $(\mathbf{d},\mathbf{r})$ is a smooth structure.
\end{proof}

\begin{lemma}\label{lem:CT-gcd-leaf-nodes}
Any smooth arithmetical structure $(\textbf{d}, \textbf{r})$ on 
$\CT{p}{s}$ has $\gcd(r_{\ell_1}, \dots, r_{\ell_s}) = 1$.
\end{lemma}
\begin{proof}
Denote $g = \gcd(r_{\ell_1}, \dots, r_{\ell_s})$. Since $r_{\ell_j}\mid r_p$ for all $1\leq j \leq s$, this implies $g \mid r_p$. Since $r_{p-1} = d_pr_p - \sum_{j=1}^s r_{\ell_j}$, we have $g \mid r_{p-1}$. 
Consider the previous argument as the base case for induction. 
By induction, assume that $g\mid r_j$ for all $i\leq j\leq p$. Now we want to show that $g\mid r_{i-1}$.
Note $d_{i}=\frac{r_{i+1}+r_{i-1}}{r_{i}}$ which simplifies to
$r_{i-1}=d_ir_{i}-r_{i+1}$ (where we take $r_{p+1}$ to be $0$).
By induction hypothesis $g\mid r_{i+1}$ and $g\mid r_i$, which implies that $g\mid r_{i-1}$.
Thus, $g\mid r_i$ for all $1\leq i\leq p$. Thus $g$ divides every label on the graph $\CT{p}{s}$,cimplying that $g=1$. 
\end{proof}

\section{Smoothing and Subdivision}\label{sec:smothsub}
In this section, we present two operations that have proven useful in the enumeration of arithmetical structures on paths and cycles \cite{braun_2018}, bidents \cite{archer_2020}, $E_n$-graphs \cite{vetter21}, as well as other works involving arithmetical structures \cite{CorralesValencia,DiazLopezEtAl,glass20,keyes21,WangHou}.

\subsection{Smoothing Arithmetical Structures on Coconut Trees}
We can smooth arithmetical structures on $\CT{p}{s}$ in a similar way as developed in Archer et al.~\cite{archer_2020} for smoothing arithmetical structures on $\CT{p}{2}$. 

Before stating the next definition we remark that in graph theory the use of the word ``smoothing a vertex'' refers to the replacement  of a degree two vertex by an edge connecting the two neighbors. We use the same naming convention for this operation, but stress that it does not necessarily imply ``smoothness'' in the sense of arithmetical structures as divisibility conditions are generally not satisfied if you smooth at a vertex $v$ with $d_v\neq 1$.

\begin{definition}\label{def:CT-smoothing-path-vertices}
Let $p,s \in \N$. For ${2}\leq i\leq p-1$, a \textbf{smoothing at the vertex $v_i$ of degree $2$} when $d_i = 1$ is defined as an operation that takes an arithmetical structure $(\textbf{d}, \textbf{r})$ on $\CT{p}{s}$ and returns an arithmetical structure $(\textbf{d}', \textbf{r}')$ on $\CT{p-1}{s}$, where the components of the vectors $\textbf{d}'$, $\textbf{r}'$ are as follows:
\begin{align*}
r_j'    &=  \begin{cases}
            r_j         &   j \in \{1,2,\dots,i-1,\ell_1, \ell_2, \dots, \ell_s\} \\
            r_{j+1}     &   j \in \{i,i+1, \dots, p-1\}
            \end{cases} \\
            \intertext{and}
d_j'    &=  \begin{cases}
            d_j         &   j \in \{1, \dots, i - 2,\ell_1, \ell_2, \dots, \ell_s\} \\
            d_j - 1     &   j = i - 1 \\
            d_{j+1} - 1 &   j = i \\
            d_{j+1}     &   j \in \{i + 1, \dots, p-1\}.
            \end{cases}
\end{align*}
\end{definition}
The requirement that $d_i=1$ is present to ensure that the resulting vectors $(\mathbf{d}',\mathbf{r}')$ form a structure on $\CT{p-1}{s}$ as shown in Proposition \ref{prop:smoothing}. Before we present this theorem, we illustrate Definition~\ref{def:CT-smoothing-path-vertices}.
\begin{figure}
    \centering
    \begin{tikzpicture}
    \tikzstyle{every node}=[draw,circle,fill=white,minimum size=4pt,inner sep=2pt]
    \node[](v1) at (6,1){2};
    \node[](v2) at (6,0){2};
    \node[](v3) at (6,-1){1};
    \node[](v4) at (1,0){1};
    \node[fill=none,dashed](v5) at (2,.25){3};
    \node[](v6) at (3,0){2};
    \node[](v7) at (4,0){3};
    \node[](v8) at (5,0){4};
    \draw(v1)--(v8)--(v7)--(v6)--(v5)--(v4);
    \draw(v2)--(v8)--(v3);
    \draw[thick, ->](7,0)--(8,0);
    \end{tikzpicture}
\quad
\begin{tikzpicture}
    \tikzstyle{every node}=[draw,circle,fill=white,minimum size=4pt,inner sep=2pt]
    \node[](v1) at (6,1){2};
    \node[](v2) at (6,0){2};
    \node[](v3) at (6,-1){1};
    \node[](v5) at (2,0){1};
    \node[](v6) at (3,0){2};
    \node[](v7) at (4,0){3};
    \node[](v8) at (5,0){4};
    \draw(v1)--(v8)--(v7)--(v6)--(v5);
    \draw(v2)--(v8)--(v3);
    \end{tikzpicture}

    \caption{Smoothing the $\textbf{r}$-structure $(1,3,2,3,4,2,2,1)$ on $\CT{5}{3}$ at $v_2$ (which we highlight as a dashed vertex) to get $\textbf{r}'$-structure $(1,2,3,4,2,2,1)$ on $\CT{4}{3}$.}
    \label{fig:CT-5-3}
\end{figure}
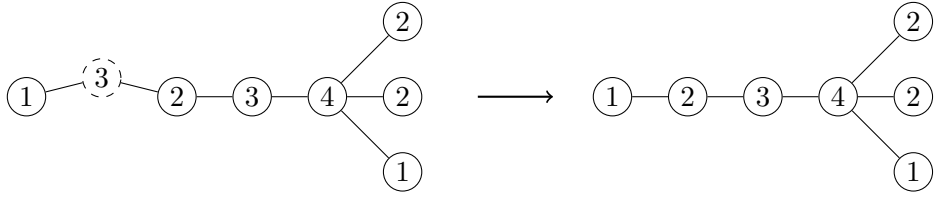
\begin{example}Consider the coconut tree $\CT{5}{3}$ with 
\[ \mathbf{d}=(3,1,3,2,2,2,2,4) \text{ and } \mathbf{r}=(1,3,2,3,4,2,2,1).\]
Since $d_2=1$, this arithmetical structure is not smooth. We can smooth at $v_2$ by removing $v_2$ and connecting $v_1$ to $v_3$, as illustrated in the right subfigure in Figure~\ref{fig:CT-5-3}. 
Note that the removal of $v_2$  yields an arithmetical structure whose entries of the $\mathbf{d}$-vector for $v_1$ and $v_3$ are both reduced by 1. Namely, the resulting structure is 
\[\mathbf{d}'=(2,2,2,2,2,2,4) \text{ and } \mathbf{r}'=(1,2,3,4,2,2,1),\] which is a smooth arithmetical structure on $\CT{4}{3}$.
\end{example}

\begin{proposition}\label{prop:smoothing}
Let $p\geq2$ and $s\geq1$ be integers and let $(\mathbf{d},\mathbf{r})$ be an arithmetical structure on $\CT{p}{s}$. If $d_i=1$ for some $1\leq i\leq p-1$, then $(\mathbf{d}',\mathbf{r}')$ resulting from smoothing vertex $v_i$ is a valid arithmetical structure on $\CT{p-1}{s}$.
\end{proposition}
\begin{proof}There are two cases to consider:\begin{itemize}
    \item For $j<i-1$ or $j>i+1$, the neighbors of $v_j$ are unchanged, so the divisibility condition still holds for $\mathbf{r}'$.
    \item For $j=i-1$, $j=i$ or $j=i+1$, we note that $d_i=1$ implies $r_i=r_{i-1}+r_{i+1}$, and since 
    $r_{i-1}| r_{i-2}+r_{i}$, we can substitute to get that $r_{i-1}|r_{i-2}+r_{i-1}+r_{i+1}$. Now note that as $r_{i-1}|r_{i-1}$, then it must be that $r_{i-1}|r_{i-2}+r_{i+1}$. An analogous argument shows that $r_{i+1}|r_{i-1}+r_{i+2}$. Thus ensuring that 
    removing $r_{i}$ preserves the divisibility condition on $r_{i-1}$ and $r_{i+1}$.
\end{itemize}
Therefore $(\textbf{d}',\textbf{r}')$ is an arithmetical structure on $\CT{p-1}{s}$, as claimed.
\end{proof}

We now give the definition for smoothing at a degree one vertex. Note that the smoothing process on a degree one vertex  removes either $v_1$ or one of the leaf nodes.

\begin{definition}\label{def:CT-smoothing-leaf-vertices}
Let $p,s\in\N$. A \textbf{smoothing at the vertex $v_{i}$} of degree $1$ for $i=1$ or $i\in\{\ell_1,\dots,\ell_s\}$ when $d_{i} = 1$ is defined as an operation that takes an arithmetical structure $(\textbf{d}, \textbf{r})$ on $\CT{p}{s}$ and returns an arithmetical structure $(\textbf{d}', \textbf{r}')$ on $\CT{p-1}{s}$ or $\CT{p}{s-1}$, where the components of the vectors $\textbf{d}'$, $\textbf{r}'$ are as follows:

\noindent For $v_1$, we have
\begin{align*}
r_j'    &=  \begin{cases}
            r_{j+1}        &\mbox{if }   j \in \{1,2,\dots p-1\}\\
            r_j            &\mbox{if }   j\in \{\ell_1,\dots,\ell_{s}\} \\
            \end{cases} \\
            \intertext{and}
d_j'    &=  \begin{cases}
            d_2 - 1     &\mbox{if }   j = 1 \\
            d_{j+1}         &\mbox{if }   j \in \{2,3,4,\dots,p-1\}\\
            d_j&\mbox{if } j \in\{\ell_1,\dots,\ell_{s}\} \\
            \end{cases}
\end{align*}
and for a leaf vertex $v_{\ell_i}$, we have 
\begin{align*}
r_j'    &=  \begin{cases}
            r_j         &\mbox{if }   j \in \{1,2,\dots,p,\ell_1,\dots,\ell_{i-1}\} \\
            r_{j+1}     &\mbox{if }   j \in \ell_{i},\ell_{i+1}, \dots, \ell_{s-1}\}
            \end{cases} \\
            \intertext{and}
d_j'    &=  \begin{cases}
            d_j         &\mbox{if }   j \in \{1,2,\dots,p-1,\ell_1,\dots,\ell_{i-1}\} \\
            d_j-1       &\mbox{if }   j = p\\
            d_{j+1}     &\mbox{if }   j \in \{\ell_{i},\ell_{i+1}\dots,\ell_{s-1}\}.\\
            \end{cases}
\end{align*}
\end{definition}

The following example illustrates the smoothing process when $d_1=1$ or $d_{\ell_i}=1$.
\begin{example}Consider $\CT{4}{4}$ with $\mathbf{d}=(1,3,2,3,2,2,3,1)$ and $\mathbf{r}=(2,2,4,6,3,3,2,6)$.
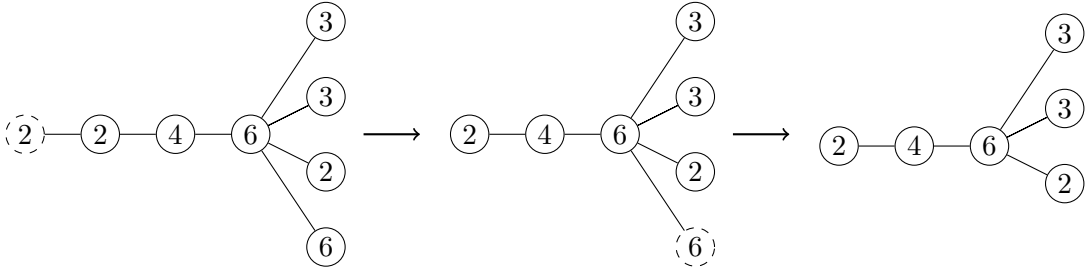
\begin{figure}
    \centering
    \begin{tikzpicture}
    \tikzstyle{every node}=[draw,circle,fill=white,minimum size=4pt,inner sep=2pt]
    \node[](v1) at (5,1.5){3};
    \node[](v2) at (5,.5){3};
    \node[](v3) at (5,-.5){2};
    \node[dashed,fill=none](v4) at (1,0){2};
    \node[](v5) at (2,0){2};
    \node[](v6) at (3,0){4};
    \node[](v7) at (4,0){6};
    \node[](v8) at (5,-1.5){6};
    \draw(v1)--(v7)--(v6)--(v5)--(v4);
    \draw(v8)--(v7)--(v2)--(v7)--(v3);
    \draw[thick, ->](5.5,0)--(6.25,0);
    \end{tikzpicture}\quad
    \begin{tikzpicture}
    \tikzstyle{every node}=[draw,circle,fill=white,minimum size=4pt,inner sep=2pt]
    \node[](v1) at (5,1.5){3};
    \node[](v2) at (5,.5){3};
    \node[](v3) at (5,-.5){2};
    \node[](v5) at (2,0){2};
    \node[](v6) at (3,0){4};
    \node[](v7) at (4,0){6};
    \node[dashed,fill=none](v8) at (5,-1.5){6};
    \draw(v1)--(v7)--(v6)--(v5);
    \draw(v8)--(v7)--(v2)--(v7)--(v3);
    \draw[thick, ->](5.5,0)--(6.25,0);
    \end{tikzpicture}\quad
    \begin{tikzpicture}
    \tikzstyle{every node}=[draw,circle,fill=white,minimum size=4pt,inner sep=2pt]
    \node[](v1) at (5,1.5){3};
    \node[](v2) at (5,.5){3};
    \node[](v3) at (5,-.5){2};
    \node[](v5) at (2,0){2};
    \node[](v6) at (3,0){4};
    \node[](v7) at (4,0){6};
    \node[color=white](v8) at (5,-1.5){};
    \draw(v1)--(v7)--(v6)--(v5);
    \draw(v7)--(v2)--(v7)--(v3);
    \end{tikzpicture}
    \caption{{Smoothing the \textbf{r}-structure $(2,2,4,6,3,3,2,6)$ on $\CT{4}{4}$ at vertex $v_1$ followed by smoothing at vertex $v_{\ell_4}$}. The right-most figure is the result of these smoothing operations.}
    \label{fig:CT-4-3}
\end{figure}\\
 For this example, we can first remove $r_1$ and relabel all of the path vertices to $r_{i}'=r_{i+1}$, reducing $d_2$ by 1. Then remove $r_{\ell_4}$, which reduces $d_4$ (or $d'_3$) by 1, no relabeling is required for this removal. 
This gives $\mathbf{r}'=(2,4,6,3,3,2)$ and $\mathbf{d}'=(2,2,2,2,2,3)$, which is a smooth 
arithmetical structure on $\CT{3}{3}$, illustrated on the right in Figure~\ref{fig:CT-4-3}. 
\end{example}

\begin{proposition}
Let $p,s\in\N$, and let $(\mathbf{d},\mathbf{r})$ be an arithmetical structure on $\CT{p}{s}$. If, for some $i \in \{1,\ell_{1},\dots,\ell_{s}\}$, we have $d_{i} = 1$, then $(\mathbf{d}',\mathbf{r}')$ resulting from smoothing at vertex $v_i$ is a valid arithmetical structure on $\CT{p-1}{s}$ if $i=1$ and $\CT{p}{s-1}$ if $i\in\{\ell_{1},\dots,\ell_{s}\}$.
\end{proposition}
\begin{proof}
There are two cases to consider:\begin{itemize}
    \item If $d_1=1$, then $r_1=r_2$. We remove $v_1$ and relabel the vertices so that $r_j'=r_{j+1}$ for $j\in\{1,2,\ldots, p-1\}$. Since $r_2|r_1+r_3$ and $r_1=r_2$ then $r_2|r_3$, which implies that $d_1'=d_2-1$. All other divisibility conditions remain unchanged, hence the result is a valid arithmetical structure.
    \item If $d_{\ell_j}=1$ then $r_{\ell_j}=r_p$ for some leaf vertex $\ell_j$. We remove that leaf vertex. Since $r_{\ell_j}\equiv0\mod{r_p}$, subtracting $r_{\ell_j}$ does not change the divisibility condition at $r_p$. All other divisibility conditions remain unchanged, hence the result is a valid arithmetical structure. We also have that $d_p'=d_p-1$.
\end{itemize}
Therefore $(\textbf{d}',\textbf{r}')$ is an arithmetical structure on $\CT{p-1}{s}$ in the first case, and on $\CT{p}{s-1}$ in second case, as claimed.
\end{proof}

We use the following concept of ancestor as in \cite{archer_2020,glass20}.

\begin{definition}
Fix $p,s\in\N$. Let $1 \leq q \leq p$ and $1 \leq t \leq s$. An arithmetical structure $(\textbf{d}', \textbf{r}')$ on $\CT{q}{t}$ is called an \textbf{ancestor} if it is obtained from a sequence of smoothing operations on an arithmetical structure $(\textbf{d}, \textbf{r})$ on $\CT{p}{s}$. We call $(\textbf{d}, \textbf{r})$ a \textbf{descendant} of $(\textbf{d}', \textbf{r}')$ if and only if $(\textbf{d}', \textbf{r}')$ is an ancestor of  $(\textbf{d}, \textbf{r})$.
\end{definition}

\begin{lemma}\label{lem:CT-unique-smooth-ancestor}
Every arithmetical structure on $\CT{p}{s}$ with $d_{\ell_1}, \dots, d_{\ell_s} \geq 2$ has a unique smooth arithmetical structure on $\CT{q}{s}$ as an ancestor for some $q$ satisfying $1\leq q\leq p$.
\end{lemma}
\begin{proof}
The proof is analogous to that of Lemma 2.6 in \cite{archer_2020} and so we omit it.
\end{proof}

\subsection{Subdividing Arithmetical Structures on Coconut Trees}

Next, we move on to the subdivision operation. Recall that a smoothing operation can remove a vertex of degree 2 or less if their associated $d$-value is $1$, which can only happen if the $r$-value of the vertex is equal to the sum of its neighbors. 
Then, a subdivision operation can be thought of as an inverse of the smoothing operation, as it always constructs a new vertex with an $\mathbf{r}$-labeling that is equal to the sum of its neighbors. 
Subdivisions provide the foundation for enumerating arithmetical structures on coconut trees, as they reduce the problem to counting the number of smooth arithmetical structures on coconut trees.

\begin{definition}\label{def:subdivision}
Let $p,s\in\N$. A \textbf{subdivision at the vertex $v_i$ for $i \in \{1, 2, \dots, p\}$} (or ``at position $i$'') is defined as an operation that takes an arithmetical structure $(\textbf{d}, \textbf{r})$ on $\CT{p}{s}$ and returns an arithmetical structure $(\textbf{d}', \textbf{r}')$ on $\CT{p+1}{s}$, where the components of the vectors $\textbf{d}'$, $\textbf{r}'$ are as follows: 

\noindent If $i>1$, then
\begin{align*}
r_j'    &=  \begin{cases}
            r_j             & \mbox{if }  j \in \{1, \dots, i - 1,\,\ell_1,\dots,\ell_s\} \\
            r_{j-1} + r_{j} & \mbox{if }  j = i \\
            r_{j-1}         & \mbox{if }  j \in \{i + 1, i + 2, \dots, p,p+1\}
            \end{cases}\\
            \intertext{and}
d_j'    &=  \begin{cases}
            d_j         & \mbox{if }  j \in \{1, \dots, i - 2,\ell_1,\dots,\ell_s\} \\
            d_j + 1     &\mbox{if }   j = i - 1 \\
            1           &\mbox{if }   j = i \\
            d_{j-1} + 1 & \mbox{if }  j = i + 1 \\
            d_{j-1}     & \mbox{if }  j \in \{i + 2, \dots, p,p+1\}.
            \end{cases}
\end{align*}

\noindent If $i=1$, then
\begin{align*}
r_j'    &=  \begin{cases}
            r_1         & \mbox{if }  j=1\\
            r_{j-1}     &\mbox{if }   j\in\{2,3,\ldots, p+1\}\\
            r_j         & \mbox{if }  j\in\{\ell_1,\dots,\ell_s\}\\
            \end{cases}\\
            \intertext{and}
d_j'    &=  \begin{cases}
            1           & \mbox{if }  j=1\\
            d_{j-1} + 1     &\mbox{if }   j = 2 \\
            d_{j-1}         & \mbox{if }  j \in \{3,4,\dots, p,p+1\} \\
            d_j&\mbox{if }j \in\{\ell_1,\dots,\ell_s\}.\\
            \end{cases}
\end{align*}
\end{definition}

In short, a subdivision at position $i$ takes the sum of $r_{i-1}$ and $r_{i}$ and assigns that value to a new vertex  $r_{i}'$.
We illustrate the definition through the following example.
\begin{example}
In Figure~\ref{fig:CT-3-2-subdivision}, we consider an arithmetical structure on $\CT{3}{2}$ and subdivide at position $2$ to get the structure on the right-hand side of the figure.
Note the result is an arithmetical structure on $\CT{4}{2}$. 
In Figure~\ref{fig:d5-subdivision-end-node}, we consider the same arithmetical structure on $\CT{3}{2}$, and do a subdivision at position $1$. This takes the beginning vertex with labeling $1$ and assigns the same value $1$ to a new vertex at the beginning of the path. Note the result is an arithmetical structure on $\CT{4}{2}$.
\begin{figure}
    \centering
    \begin{tikzpicture}
    \tikzstyle{every node}=[draw,circle,fill=white,minimum size=5pt,inner sep=2pt]
    \node[](v1) at (3,1){2};
    \node[](v2) at (3,-1){7};
    \node[](v3) at (0,0){1};
    \node[](v4) at (1,0){5};
    \node[](v5) at (2,0){14};
    \draw(v1)--(v5)--(v4)--(v3);
    \draw(v2)--(v5);
    \draw[thick, ->](3.75,0)--(5.25,0);
    \node[](v1) at (9,1){2};
    \node[](v2) at (9,-1){7};
    \node[](v3) at (6,0){1};
    \node[](v4) at (7,0){5};
    \node[](v5) at (6.5,.75){6};
    \node[](v6) at (8,0){14};
    \draw(v1)--(v6)--(v4)--(v5)--(v3);
    \draw(v2)--(v6);
    \end{tikzpicture}
    \caption{Subdividing at position $2$ on an arithmetical structure on $\CT{3}{2}$.}
    \label{fig:CT-3-2-subdivision}
\end{figure}
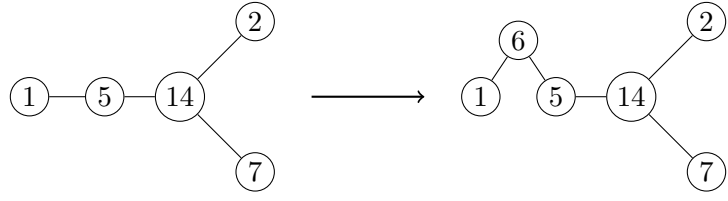

\begin{figure}
    \centering
    \begin{tikzpicture}
    \tikzstyle{every node}=[draw,circle,fill=white,minimum size=5pt,inner sep=2pt]
    \node[](v1) at (3,1){2};
    \node[](v2) at (3,-1){7};
    \node[](v3) at (0,0){1};
    \node[](v4) at (1,0){5};
    \node[](v5) at (2,0){14};
    \draw(v1)--(v5)--(v4)--(v3);
    \draw(v2)--(v5);
    \draw[thick, ->](3.75,0)--(5.25,0);
    \node[](v1) at (10,1){2};
    \node[](v2) at (10,-1){7};
    \node[](v3) at (9,0){14};
    \node[](v4) at (8,0){5};
    \node[](v5) at (7,0){1};
    \node[](v6) at (6.25,.75){1};
    \draw(v1)--(v3)--(v4)--(v5)--(v6);
    \draw(v2)--(v3);
    \end{tikzpicture}
    \caption{Subdividing at position $1$ on an arithmetical structure on $\CT{3}{2}$.}
    \label{fig:d5-subdivision-end-node}
\end{figure}
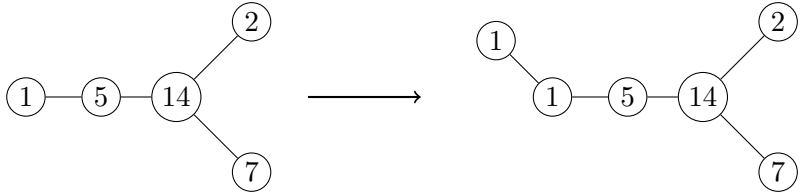
\end{example}

Next, we prove that subdivision results in a valid arithmetical structure.
\begin{proposition}
Let $p,s\in\N$, and let $(\mathbf{d},\mathbf{r})$ be an arithmetical structure on $\CT{p}{s}$. We have that $(\mathbf{d}',\mathbf{r}')$ resulting from subdividing vertex $i\in\{1,\dots,p\}$ is a valid arithmetical structure on $\CT{p+1}{s}$.
\end{proposition}
\begin{proof}
Note that the gcd condition is unchanged by adding a new vertex, so we only need to check divisibility. There are two cases:
\begin{itemize}
\item For $i=1$, note that $r_1'=r_2'=r_1$, so $r_1'\mid r_2'$. Since $r_1\mid r_2$ and $r_3'=r_2$, then $r_2'\mid r_1'+r_3'$. For all other vertices, the divisibility condition does not change, so $(\mathbf{d}',\mathbf{r}')$ is an arithmetical structure on $\CT{p+1}{s}$.
\item For $i>1$, the only divisibility conditions that have changed are related to the new vertex and its two adjacent vertices. Note that $r_{i-1}\mid r_{i-2}+r_i$, hence $r_{i-1}\mid r_{i-2}+(r_{i-1}+r_i)$, which is equivalent to $r_{i-1}'\mid r_{i-2}'+r_i'$. Regarding $r_i'$, it divides the sum of the labels on its two neighbors as it is defined as exactly that sum, that is, $r_i'=r_{i-1}+r_{i}=r_{i-1}'+r_{i+1}'.$ Finally, a similar argument to the one for $r_{i-1}'$ shows that $r_{i+1}'$ (which equals $r_i$) divides the sum of its neighbors as the only change in the sum of its neighbors is the addition of $r_i$. So $(\mathbf{d}',\mathbf{r}')$ is a valid arithmetical structure on $\CT{p+1}{s}$.\qedhere
     \end{itemize}
\end{proof}

\section{Counting Arithmetical Structures on Coconut Trees}\label{sec:4}
We now focus our attention on enumerating arithmetical structures on coconut trees.
To begin, we define a subdivision sequence.

\begin{definition}
Let $(\textbf{d}^0, \textbf{r}^0)$ be an arithmetical structure on $\CT{p}{s}$, with $p,s\in\N$. A sequence of positive integers $\textbf{b} = (b_1, b_2, \dots, b_k)$ is a \textbf{valid subdivision sequence} if its entries satisfy $1 \leq b_{i} \leq p+i-1$ for $i \in \{1, 2, \dots, k\}$.

The arithmetical structure \textbf{Sub($(\textbf{d}^0, \textbf{r}^0)$, \textbf{b})} on $\CT{p}{s}$ is inductively defined as follows. Let $(\textbf{d}^i, \textbf{r}^i)$ be the arithmetical structure on $\CT{p+i}{s}$ obtained from the arithmetical structure $(\textbf{d}^{i-1}, \textbf{r}^{i-1})$ on $\CT{p+i-1}{s}$ by subdividing at the vertex $v_{b_i}$. Then, let $$\textbf{Sub($(\textbf{d}^0, \textbf{r}^0)$, \textbf{b})} \coloneqq (\textbf{d}^k, \textbf{r}^k) \in \CT{p+k}{s}.$$
\end{definition}

\begin{example}\label{ex:validsubdivision}
    Consider the structure $(\mathbf{d}^0,\mathbf{r}^0)\in \CT{8}{3}$ with $\mathbf{d}^0=(2,2,2,2,2,2,2,2,8,2,2)$ and $\mathbf{r}^0=(1,2,3,4,5,6,7,8,1,4,4).$ Let $\textbf{b}=(3,4,4,7)$. Note that $3\leq 8, 4\leq 9, 4\leq 10,$ and $7\leq 11$, hence $\textbf{b}$ is a valid subdivision sequence. We now construct $(\textbf{d}^i, \textbf{r}^i)$ for $i=1,2,3,4$,
    \[\begin{array}{lcl}
        \mathbf{d}^1=(2,3,\textbf{1},3,2,2,2,2,2,8,2,2)&\qquad  &\mathbf{r}^1=(1,2,\textbf{5},3,4,5,6,7,8,1,4,4)\\
        \mathbf{d}^2=(2,3,2,\textbf{1},4,2,2,2,2,2,8,2,2)&\qquad  &\mathbf{r}^2=(1,2,5,\textbf{8},3,4,5,6,7,8,1,4,4)\\
        \mathbf{d}^3=(2,3,3,\textbf{1},2,4,2,2,2,2,2,8,2,2)&\qquad  &\mathbf{r}^3=(1,2,5,\textbf{13},8,3,4,5,6,7,8,1,4,4)\\
        \mathbf{d}^4=(2,3,3,1,2,5,\textbf{1},3,2,2,2,2,8,2,2)&\qquad  &\mathbf{r}^4=(1,2,5,13,8,3,\textbf{7},4,5,6,7,8,1,4,4).\\
    \end{array}\]
    Therefore, $\textbf{Sub($(\textbf{d}^0, \textbf{r}^0)$, \textbf{b})}$ is the structure 
    \[\left((2,3,3,1,2,5,1,3,2,2,2,2,8,2,2), (1,2,5,13,8,3,7,4,5,6,7,8,1,4,4)\right)\in \CT{12}{3}.\]
\end{example}
Next we define the ballot numbers, a generalization of the Catalan numbers, where the $n$th Catalan number (for $n\geq 1$) is defined by \cite[\href{https://oeis.org/A000108}{A000108}]{OEIS}: 
\[C_{n}=\frac{1}{n+1}\binom{2n}{n}.\]
The ballot numbers, which we denote by $B(n, k)$, count the number of lattice paths from $(0, 0)$ to $(n, k)$ that do not cross above the line $y = x$. For more on ballot numbers see \cite{Carlitz:1972}. 
\begin{definition}\label{def:ballot-numbers}
For all $n, k \in \N \cup \{0\}$, with $0\leq k\leq n$, define the ballot numbers \cite[\href{https://oeis.org/A009766}{A009766}]{OEIS} by
\begin{align}\label{ballot numbers}
B(n, k)     &\coloneqq \frac{n - k + 1}{n + 1} \binom{n + k}{n}.
\end{align}
\end{definition}

\begin{example}\label{ex:ballots}
    Consider the  lattice path from $(0,0)$ to $(8,4)$ depicted in Figure \ref{fig:0085}.
    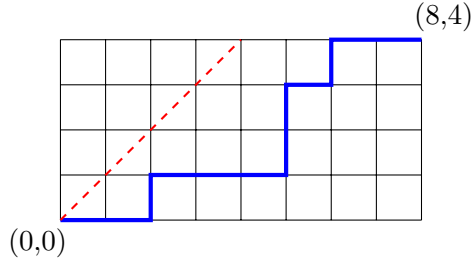
\begin{figure}
        \centering
        \begin{tikzpicture}[scale=0.6]
        \draw[step=1.0,black,thin] (0,0) grid (8,4);
        \draw[ultra thick,blue](0,0)--(2,0)--(2,1)--(5,1)--(5,3)--(6,3)--(6,4)--(8,4);
  \draw[dashed, red, thick](0,0)--(4,4);
  \node[] at (-.5,-.5){(0,0)};
  \node[] at (8.5,4.5){(8,4)};
\end{tikzpicture}
        \caption{A lattice path starting at $(0,0)$ and ending at $(8,4)$, {which is associated to the subdivision sequence $\textbf{b}=(3,4,4,7)$.} }
        \label{fig:0085}
    \end{figure}
    We now propose a subdivision sequence that is in bijection with these lattice paths. We want a vector with $k$ parts (subdivisions) so that the value of every part satisfies $1 \leq b_{i} \leq p+i-1$.
    \par We do so by the following process:
\begin{enumerate}
    \item First, identify every ``north'' step we take and place  the corresponding $x$-coordinate in a vector $(x_1,x_2,\dots,x_k)$. \item Then, subtract each value of the vector from $n+1$.
    \item Place the resulting values in ascending order in a new vector, $\mathbf{b}=(b_1,b_2,\dots,b_k)$.
\end{enumerate}

    \par Let's return to Figure~\ref{fig:0085}, and find its corresponding subdivision sequence. The 4 north steps come at the $x$-coordinates $2$, 5, 5, and 6. So we get the vector $(2,5,5,6)$. Then, we subtract $9$ from them to get $9-2=7,9-5=4,9-5=4$ and $9-6=3$ and arrange them in ascending order to get $\mathbf{b}=(3,4,4,7)$. This is a valid subdivision sequence because each $b_i$ is between 1 and $8+i-1$.
    Note that each lattice path takes exactly $k$ north steps. Due to the line $y=x$, there can never be more than 1 north step at $x=1$, 2 north steps at $x=2$, and so on. The values also can never go above $n$. These are the exact conditions for a subdivision sequence with $k$ subdivisions, and max value $n$.
    Hence, the number of vectors $\textbf{b}$ which are valid subdivision sequences with $k$ subdivisions and max value $n$ is the same as the number of lattice paths from $(0,0)$ to $(n,k)$ that do not go above $y=x$.
\end{example}

The next result is used in counting the number of arithmetical structures on $\CT{p}{s}$ as we find a bijection between the subdivision sequences on smaller coconut trees and non-decreasing sequences, similar to the method used in Archer et al. \cite{archer_2020}.

\begin{proposition}\label{prop:CT-counting-descendants}
Fix $s \geq 2$. Let $1 \leq i \leq p$ and let $(\textbf{d}, \textbf{r})$ be any smooth arithmetical structure on $\CT{i}{s}$. 
The number of arithmetical structures on $\CT{p}{s}$ that are descendants of $(\textbf{d}, \textbf{r})$ is $B(p - 1, p- i)$.
\end{proposition}
\begin{proof}
The proof follows in a similar manner to Proposition 2.11 in \cite{archer_2020}.
\end{proof}

Next, we count the number of arithmetical structures on $\CT{p}{s}$ with non-smoothable leaf nodes, that is, when $d_{\ell_j} > 1$ for $j \in \{1,2, \dots, s\}$.

\begin{corollary}\label{cor:CT-number-of-AS-leaf-smooth}
Fix $p\geq 1,  s \geq 2$. The number of arithmetical structures on $\CT{p}{s}$ such that all of the leaf vertices {$v_{\ell_j}$ for $j \in\{1,2,\ldots, s\}$} cannot be smoothed is given by \[\sum_{i=1}^{p} B(p - 1, p-i)|\sarith{\CT{i}{s}}|.\]
\end{corollary}
\begin{proof}
Since Proposition~\ref{prop:CT-counting-descendants} shows that each smooth structure on $\CT{i}{s}$ has $B(p - 1, p-i)$ descendent arithmetical structures on $\CT{p}{s}$, we iterate through every smooth structure on $\CT{i}{s}$ from $i = 1$ to $i = p$ and count the number of descendent arithmetical structures to get the total count of structures on $\CT{p}{s}$ such that the leaf vertices cannot be smoothed. Thus, the number of arithmetical structures on $\CT{p}{s}$ with $d_{\ell_1}, \dots, d_{\ell_k} \geq 2$ is given by
\[
  \sum_{i=1}^{p} B(p - 1, p-i)|\sarith{\CT{i}{s}}|.\qedhere
\]
\end{proof}

\begin{example}
Referring back to Figure~\ref{fig:0085}, each lattice path corresponds to a unique descendant of a smooth arithmetical structure $(\textbf{d}, \textbf{r})$ on $\CT{5}{s}$. The descendant structures are on $\CT{9}{s}$. There are \[B(p-1,p-i)=B(9-1,9-5)=B(8,4)=275\] such descendants.
\end{example}

\begin{proposition}\label{prop:CT-counting-AS-fork-vertices}
Let $p \geq 1$. The number of arithmetical structures on $\CT{p}{1}$ such that {the leaf vertex $v_{\ell_1}$} cannot be smoothed is given by $|\arith{\mathcal{P}_{p+1}}| - |\arith{\mathcal{P}_{p}}| = C_{p} - C_{p-1}$, where $C_n$ is the $n$th Catalan number.
\end{proposition}
\begin{proof}
As proven by Braun et al. in~\cite[Lemma 1]{braun_2018}, 
arithmetical structures on a path satisfy that the $r$-values at the first and last vertex of the path must be $1$. 
Thus, if $v_{\ell_1}$ can be smoothed, then $r_{\ell_1}=r_p=1$. The arithmetical structures on $\CT{p}{1}=\mathcal{P}_{p+1}$ with $r_{\ell_1}=r_p=1$ are in bijection with arithmetical structures on $\mathcal{P}_p$ via smoothing at $v_{\ell_1}$. Hence, there are $|\arith{\mathcal{P}_{p}}| = C_{p-1}$ many of them. Therefore, there are $|\arith{\mathcal{P}_{p+1}}| - |\arith{\mathcal{P}_{p}}|=C_p-C_{p-1}$ many arithmetical structures on $\CT{p}{1}$ with the property that the end leaf vertex cannot be smoothed. 
\end{proof}

Note that for $s = 0$, there are no leaf vertices that can be smoothed. Thus, the number of arithmetical structures on $\CT{p}{0} = \mathcal{P}_{p}$ is $|\arith{\mathcal{P}_p}| = C_{p-1}$. 
Next, we prove the main result of the paper, providing a count for the number of arithmetical structures on $\CT{p}{s}$ in terms of smooth arithmetical structures.

\begin{theorem}{\label{thm:CT-number-of-AS}}
If $p\geq 1, s\geq 2$ then the number of arithmetical structures for the coconut tree $\CT{p}{s}$ is given by
\begin{align}\label{eq:recursion with A}
|\arith{\CT{p}{s}}|  &=  \sum_{j=0}^s \binom{s}{j} A(p+s - j, s - j)
\end{align}
where
\begin{align*}
    A(x, 0) &= C_{x-1} \mbox{ for $x\geq 1$,}\\
    A(x, 1) &= C_{x-1} - C_{x-2}\mbox{ for $x\geq 2$, and}\\
    A(x, y) &= \sum_{i=y+1}^{x} B(x-y-1, x-i)|\sarith{\CT{i-y}{y}}|\mbox{ for $x\geq 3, y\geq2$.}
\end{align*}
\end{theorem}

\begin{proof}
We begin by counting the number of arithmetical structures on $\CT{p}{s}$ by enumerating over the number of arithmetical structures on $\CT{p}{s}$ that have $j$ leaf vertices (among the $s$ vertices $v_{\ell_1}, v_{\ell_2},\dots,v_{\ell_s}$) that can be smoothed for $0 \leq j \leq  s$.

For $0\leq j \leq s-2$, by Corollary~\ref{cor:CT-number-of-AS-leaf-smooth}, the number of arithmetical structures on $\CT{p}{s-j}$ such that all $s-j$ leaf vertices cannot be smoothed is given by 
\begin{align*}
    &\qquad\sum_{k=1}^{p} B(p - 1, p-k)|\sarith{\CT{k}{s-j}}|\\
    &=\sum_{i=s-j+1}^{p+(s-j)} B(p-1,p-i+s-j)|\sarith{\CT{i-(s-j)}{s-j}}|\\
    &=A(p+s-j, s - j).
\end{align*} 
Given one such structure $(\textbf{d},\textbf{r})$ in $\CT{p}{s-j}$, we can construct $\binom{s}{j}$ structures in $\CT{p}{s}$ by choosing $j$ leaf vertices and setting their $r$-value to $r_p$, placing the labels $r_{\ell_1},r_{\ell_2},\ldots, r_{\ell_{s-j}}$ (none of which equals $r_p$) in the remaining $s-j$ leaf vertices in the order listed, and keeping the values of $r_1,r_2,\ldots, r_p$ intact in the first $p$ vertices.
Thus, the total number of arithmetical structures with $j$ vertices that can be smoothed is given by $\binom{s}{j} A(p+s-j, s - j)$. Taking the sum over the possible values of $j$ gives all but the top two terms in  \eqref{eq:recursion with A}.
We consider those values for $j$ next.

If $j=s-1$, then $\CT{p}{s-j}=\CT{p}{1}$. 
By Proposition~\ref{prop:CT-counting-AS-fork-vertices}, 
the number of arithmetical structures on $\CT{p}{1}$ such that the leaf vertex $v_{\ell_1}$ cannot be smoothed is $C_{p}-C_{p-1}=A(p+1,1)$.  Similar to the previous case, given one such structure $(\textbf{d},\textbf{r})$ in $\CT{p}{1}$, we can construct $\binom{s}{s-1}$ structures in $\CT{p}{s}$ by choosing $s-1$ leaf vertices and setting their $r$-value to $r_p$, placing the label $r_{\ell_1}$ (that is not equal to $r_p$) in the remaining leaf vertex, and keeping the values of $r_1,r_2,\ldots, r_p$ intact in the first $p$ vertices.
 
If $j = s$, then $\CT{p}{s-j}=\CT{p}{0}$. Hence, the number of arithmetical structures on $\CT{p}{0}$ with no leaf vertices that can be smoothed is simply the number of structures on $\CT{p}{0}=\mathcal{P}_p$, which is given by $C_{p-1}=A(p,0)$. Similar to previous cases, given one such structure, we can place the label $r_p$ on every one of the $s$ leaf vertices of $\CT{p}{s}$, to get a structure on $\CT{p}{s}$. This gives the final correspondence to get that 
\begin{align*}
|\arith{\CT{p}{s}}|  &=  \sum_{j=0}^s \binom{s}{j} A(p+s - j, s- j).\qedhere
\end{align*}
\end{proof}

We now confirm the count for the number of arithmetical structures on a coconut tree $\CT{p}{2}$ as given in \cite[Theorem 2.12]{archer_2020}.
\begin{corollary}
If $p\geq 2$, then 
\[|\arith{\CT{p}{2}}| = 2C_{p} - C_{p-1} + \sum_{i=3}^{p+2} B(p - 1, p+2 - i) |\sarith{\CT{i-2}{2}}|.\]
\end{corollary}
\begin{proof}
By Theorem~\ref{thm:CT-number-of-AS} we have 
\begin{align*}
|\arith{\CT{p}{2}}|  &=  \sum_{j=0}^2 \binom{2}{j} A(p+2 - j, 2 - j), \\
&=\binom{2}{0} A(p+2, 2) + \binom{2}{1} A(p + 1, 1) + \binom{2}{2} A(p, 0) \\
&=\left(\sum_{i=3}^{p+2}B(p-1,p+2-i)|\sarith{\CT{i-2}{2}}|\right) + 2 (C_{p}-C_{p-1}) + C_{p-1}\\
&=2C_{p} - C_{p-1} + \sum_{i=3}^{p+2} B(p -1, p+2-i) \cdot |\sarith{\CT{i-2}{2}}|.\qedhere
\end{align*}
\end{proof}
The following is a derivation of the number of arithmetical structures on a star graph. Let $\mathcal{S}_{s}$ denote the star graph on $s+1$ vertices, which is equivalent to $\CT{1}{s}$. The next result shows that the number of arithmetical structures on a star graph can be computed using the number of smooth arithmetical structures on smaller star graphs.
\begin{corollary}
    If $p=1$ and $s\geq2$, then $|\arith{\CT{1}{s}}|=1+\displaystyle\sum_{j=0}^{s-2}\binom{s}j|\textbf{SArith}(\mathcal{S}_{s-j})|$.
\end{corollary}
\begin{proof}
Using Theorem~\ref{thm:CT-number-of-AS}, we have
\begin{align*}
    |\arith{\CT{1}{s}}|&=\sum_{j=0}^{s}\binom{s}{j} A(s+1-j,s-j)\\
    &\hspace{-1in}=\binom{s}{s}A(1,0)+\binom{s}{s-1}A(2,1)+\sum_{j=0}^{s-2}\binom{s}{j}A(s+1-j,s-j)\\
    &\hspace{-1in}=1\cdot1+s\cdot0\\
    &\hspace{-.75in}+\sum_{j=0}^{s-2}\binom{s}{j}\left(\sum_{i=s+1-j}^{s+1-j}B((i-(s-j)-1,(s+1-j)-i)|\textbf{SArith}(\CT{i-(s-j)}{s-j})|\right)\\
    &\hspace{-1in}=1+\sum_{j=0}^{j=s-2}\binom{s}jB(0,0)|\textbf{SArith}(\CT{1}{s-j})|\\
    &\hspace{-1in}=1+\sum_{j=0}^{j=s-2}\binom{s}j|\textbf{SArith}(\CT{1}{s-j})|. \qedhere
\end{align*}
\end{proof}

\subsection{Counting Smooth Arithmetical Structures}
In this subsection, we give an explicit construction of smooth arithmetical structures on coconut tree graphs and give some enumeration results. First, we define Euclidean chains, which are sequences that capture the construction of arithmetical structures on paths. As an important notation convention, in this section when we write $c=a\mod b$ we take $c$ to be the smallest nonnegative representative in the class of $a\mod b$. In particular, $c$ is always less than $b$.

\begin{definition}
A \textbf{Euclidean chain} is a sequence $\{x_i\}_{i\in \mathbb{N}}$ defined as follows: $x_1,x_2\in \mathbb{N}$ and for all $i\geq 2$,
\begin{align*}
x_{i+1}     &=  \begin{cases}
                -x_{i-1} \mod{x_i}   &  \mbox{if } x_i \neq 0, 1 \\
                0                                                                   &   \text{otherwise.}
                \end{cases}
\end{align*}
 Then, the \textbf{Euclidean chain function} $F: \N \times (\N \cup \{0\}) \to \N$ is defined as the function $F(x_1, x_2) = k$ where $k$ is the largest value of $i$ such that $x_i$ is nonzero, or the number of positive terms in the sequence $\{x_i\}$. Since $x_{i+1}=-x_{i-1} \mod{x_i}<x_i$, then the sequence eventually terminates and $F$ is well-defined.
\end{definition}

\begin{example}\label{ex:chain long}
We calculate $F(13, 60)$. Take $x_1 = 13$ and $x_2 = 60$. Then, let $x_3$ be the least residue of $-13 \pmod{60} \equiv 47 \pmod{60}$, which gives us $47 = x_3$. Likewise, let us calculate the remaining $x_i$ for $i=3,\ldots,9$, which are found as follows:
\begin{center}

$\begin{array}{llcllll}
-13 &\pmod{60} & \equiv  & 47  &\pmod{60} & \Rightarrow & x_3 = 47\\
-60 &\pmod{47} & \equiv  & 34  &\pmod{47} & \Rightarrow & x_4 = 34\\
-47 &\pmod{34} & \equiv  & 21  &\pmod{34} & \Rightarrow & x_5 = 21\\
-34 &\pmod{21} & \equiv  & 8   &\pmod{21} & \Rightarrow & x_6 = 8\\
-21 &\pmod{8}  & \equiv  & 3   &\pmod{8}  & \Rightarrow & x_7 = 3\\
-8  &\pmod{3}  & \equiv  & 1   &\pmod{3}  & \Rightarrow & x_8 = 1\\
-3  &\pmod{1}  & \equiv  & 0   &\pmod{1}  & \Rightarrow & x_9 = 0.\\
\end{array}$
\end{center}
Hence, $F(13, 60) = 8$ as $x_8$ is the last $x_i$ with a non-zero value. 
\end{example}
Next, we give the following proposition from Archer et al. \cite{archer_2020} that greatly simplifies calculations involving the Euclidean chain function $F$.
\begin{proposition}[{\cite[Lemma 3.2]{archer_2020}}]\label{prop:CT-properties-of-F}
Let $x \in \N$ and $y, k \in \N \cup \{0\}$. Then,
\begin{enumerate}
    \item $F(x + ky, y) = F(x, y)$
    \item $F(x, kx + y) = F(x, y) + k$.
\end{enumerate}
\end{proposition}

\begin{corollary}\label{cor:CT-F-xdivy}
Let $x, y \in \N$ such that $x \mid y$. Then, $F(x, y) = 1 + \frac{y}{x}$.
\end{corollary}
\begin{proof}
Since $x\mid y$, there exists some $n \in \N$ such that $nx = y$. Then, note that $F(x, y) = F(x, nx + 0) = F(x, 0) + n$, where we used property (2) of Proposition~\ref{prop:CT-properties-of-F}. Since $nx = y$ yields $n = \frac{y}{x}$ and by definition $F(x, 0) = 1$, we have $F(x, y) = 1 + \frac{y}{x}$.
\end{proof}

\begin{corollary}\label{cor:F with 1}
If $y \in \N\cup\{0\}$, then $F(1, y) = y + 1$ and if $x\in \N$, then $F(x,1)=2$.
\end{corollary}
\begin{proof}
Since $1 \mid y$, we use Corollary~\ref{cor:CT-F-xdivy} to arrive $F(1, y) = 1 + \frac{y}{1} = y + 1$. For the second statement, if a Euclidean chain starts with $\{x,1\}$ then, by definition, the next term is $0$ and the sequence is $\{x,1,0,0,\ldots\}$. Thus, $F(x,1)=2$.
\end{proof}

In what follows, we begin by assuming that we have assigned labels to the leaf vertices such that they divide the assigned label on the center vertex of a coconut tree graph.
Our main result establishes that given this initial labeling, we can construct a unique smooth arithmetical structure on a coconut tree whose path length is one less than the Euclidean chain evaluated at the sum of the leaf labels and the label of the center vertex (Corollary~\ref{cor:length of path}). Before we establish this result, we illustrate the procedure.

\begin{example}\label{ex:smooth construction}
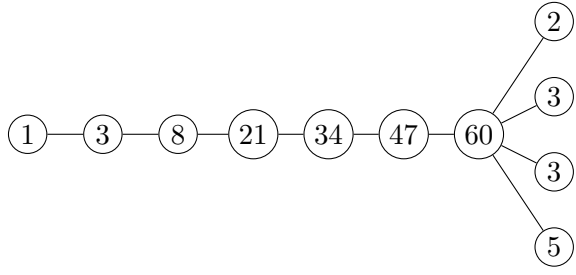
\begin{figure}
    \centering
    \begin{tikzpicture}
    \tikzstyle{every node}=[draw,circle,fill=white,minimum size=5pt,inner sep=2pt]
    \node[](v0) at (7,1.5){2};
    \node[](v1) at (7,.5){3};
    \node[](v2) at (7,-.5){3};
    \node[](v3) at (7,-1.5){5};
    \node[](v4) at (0,0){1};
    \node[](v5) at (1,0){3};
    \node[](v6) at (2,0){8};
    \node[](v7) at (3,0){21};
    \node[](v8) at (4,0){34};
    \node[](v9) at (5,0){47};
    \node[](v10) at (6,0){60};
    \draw(v0)--(v10);
    \draw(v1)--(v10);
    \draw(v2)--(v10);
    \draw(v3)--(v10);
    \draw(v4)--(v5)--(v6)--(v7)--(v8)--(v9)--(v10);
    \end{tikzpicture}
    \caption{The constructed Coconut Tree $\CT{7}{4}$.}
    \label{fig:CT-construction}
\end{figure}

Let $(c,a_1, a_2, a_3, a_4) = (60, 2, 3, 3, 5)$. Note that $\sum_{j=1}^4 a_j=13$ and by Example \ref{ex:chain long}, we have $F(13,60)=8$, hence we will construct a smooth arithmetical structure on $\CT{8-1}{4}$ with $r_7=60$ and $r_{\ell_j}=a_j$ for $j=1,2,3,4$. To determine $r_1,r_2,\ldots, r_6$, we use the positive entries of the Euclidean chain that starts $\{13,60\}$ in reverse order. By Example \ref{ex:chain long}, this chain is $\{13,60,47,34,21,8,3,1,0,0,\ldots\}$. Hence, we let 
\[\textbf{r}=(1,3,8,21,34,47,60,2,3,3,5),\]
which is illustrated in Figure~\ref{fig:CT-construction}.
\end{example}
\begin{proposition}\label{prop:CT-construction-as-unique}
For every tuple $(c, a_1, a_2, \dots, a_s) \in \N^{s+1}$ such that $\gcd(c,a_1, a_2, \dots, a_s) = 1$, $a_i \mid c$, and $a_i < c$ for $i \in \{1, 2, \dots, s\}$, there is a unique $p \geq 1$ such that there is a smooth arithmetical structure $(\textbf{d}, \textbf{r})$ on $\CT{p}{s}$ with $r_{\ell_i} = a_i$ for $i \in \{1, 2, \dots, s\}$ and ${r_p} = c$. Moreover, the arithmetical structure $(\mathbf{d},\mathbf{r})$ is also unique.
\end{proposition}
\begin{proof}
Let $(c,a_1, a_2, \dots, a_s) \in \N^{s+1}$ such that $\gcd(c,a_1, a_2, \dots, a_k) = 1$, $a_i \mid c$ and $c > a_i$ for $i \in \{1, 2, \dots, s\}$.
Let $r_{\ell_i} = a_i$ for $i \in \{1, 2, \dots, s\}$. Define the sequence $\{x_i\}_{i=1}^{\infty}$ as follows:
\
\begin{align*}
x_i     &=  \begin{cases}
              c                             &   \text{ if } i = 1, \\
             \left(-\sum_{j=1}^s a_j\right) \mod{c}
             &   \text{ if } i = 2, \\
            -x_{i-2} \mod{x_{i-1}}         &   \text{ if } i\geq 3  \text{ and } x_{i-1}, x_{i-2} \neq 0,1 \\
            0                               &   \text{otherwise},
            \end{cases}
\end{align*}
Hence, $\{x_i\}_{i=1}^\infty$ form a decreasing sequence and moreover, it is a Euclidean chain.
Let $p=F\left(c, \left(-\sum_{j=1}^s a_j\right) \mod{c}\right)$, that is, let $p$ be the number of nonzero entries of  $\{x_i\}_{i=1}^\infty$. Let {$r_{p+1-i} = x_i$}
for $i \in \{1, 2, \dots,p\}$ such that $x_i \neq 0$. Then, we verify the resulting $\textbf{r}$-vector, namely $\textbf{r}=(r_1,r_2,\ldots,r_p,a_1,a_2,\ldots,a_s)$, is an arithmetical structure.
By assumption, {$r_{\ell_i}=a_i$ and since $a_i\mid c=r_p$}, then $r_{\ell_i} \mid {r_p}$. For $v_p$, note that $r_{p-1}=\left(-\sum_{j=1}^s a_j\right) \mod{r_p}$, hence $r_p \left| \left(\sum_{j=1}^s r_{\ell_j}\right)+r_{p-1}.\right.$ For the vertices $v_2,\ldots,v_{p-1}$, we constructed the $\textbf{r}$ labels such that $r_{i+1}=-r_{i-1} \mod{r_i}$ hence $r_i\mid (r_{i-1}+r_{i+1})$. Finally, for $v_1$, since $r_1=x_p$, it is the last nonzero entry of the Euclidean chain $\{x_i\}_{i=1}^\infty$. Thus,  $x_{p+1}=0$, which means that either $x_p=1$ or $-x_{p-1}\mod{x_p}=0$. In both cases, we get that $x_p\mid x_{p-1}$, that is, $r_1\mid r_2$.
Hence, the constructed $\textbf{r}$-vector determines an arithmetical structure on $\CT{p}{s}$.

Now, we show this arithmetical structure on $\CT{p}{s}$ is smooth. Note that by assumption $r_{\ell_i} < r_p$ and $\gcd(r_p,r_{\ell_1}, r_{\ell_2}, \dots, r_{\ell_s}) = 1$ . Since the entries $(r_p,r_{p-1},\ldots, r_1)$ are consecutive entries of the Euclidean chain $\{x_i\}_{i=1}^\infty$, then $r_p>r_{p-1}>\cdots >r_1$. By Lemma~\ref{lem:CT-smooth-equivalency}, we get that $\mathbf{r}$ is a smooth arithmetical structure.
The uniqueness of $p$ (and of $\mathbf{r}$) comes from the fact that in order to have $r_p>r_{p-1}>\cdots >r_1$ and $r_i\mid (r_{i-1}+r_{i+1})$ for $i=1,2,\ldots, p$ (denoting $r_0=0$ and $r_{p+1}=\sum_{j=1}^s r_{\ell_j}$), we must have $r_{i}=-r_{i+2}\mod r_{i+1}$, hence the values of $\mathbf{r}$ must be the entries of the Euclidean chain $\{x_i\}_{i=1}^\infty$.
\end{proof}

The next corollary reduces the problem of bounding the number of smooth arithmetical structures that can be generated from only the leaf vertices to studying the Euclidean chain function $F$.

\begin{corollary}\label{cor:length of path}
Let $(c, a_1, a_2, \dots, a_s) \in \N^{s+1}$ such that $\gcd(c,a_1, a_2, \dots, a_s) = 1$, $a_i \mid c$, and $a_i<c$ for all $i\in\{1, 2, \dots, s\}$. 
The smooth arithmetical structure on the coconut tree $\CT{p}{s}$ constructed from Proposition~\ref{prop:CT-construction-as-unique} satisfies \[p = F\left(\left(\sum_{j=1}^s a_j\right),c\right) - 1.\]
\end{corollary}
\begin{proof}
Note that by the construction of the Euclidean chain $\{x_i\}$ in Proposition~\ref{prop:CT-construction-as-unique}, the number of vertices that are constructed is 
\[F\left(c, -\left(\sum_{j=1}^s a_j\right)\mod c\right)= F\left(\left(\sum_{j=1}^s a_j\right),c\right) - 1.\qedhere\]
\end{proof}

Next, we show that having only assigned labels to the leaf vertices  is not enough to guarantee a unique smooth arithmetical structure. 
\begin{example}\label{ex:CT-non-unique-AS}
Let $c = 6$, and let $(a_1,a_2,a_3,a_4,a_5)=(2,2,3,3,3)$.
Using Corollary~\ref{cor:length of path} followed by Proposition~\ref{prop:CT-properties-of-F} we have  $p=F(13,6)-1=F(1,6)-1$.
Then by Corollary~\ref{cor:F with 1}, $F(1,6)=7$, so $p=6$. On the other hand, if $c=12$, then we have $(c,a_1, a_2, a_3, a_4, a_5) = (12,2, 2, 3, 3, 3)$. Using the same results, we have $p=F(13,12)-1$, which then utilizing the fact that $F(13,12)=F(1,12)$, results in $p=12$.
Moreover, the resulting arithmetical structures on $\CT{6}{5}$ and $\CT{12}{5}$, have the path labels $1$ through $6$, and $1$ through $12$, respectively, listed in increasing order. 
\end{example}

\begin{proposition}\label{prop:CT-when-is-d_p>1}
Let $p\geq1$. Any smooth arithmetical structure $(\textbf{d}, \textbf{r})$ on $\CT{p}{s}$ has $d_p > 1$ if and only if $r_p < \sum_{j=1}^s r_{\ell_j}$.
\end{proposition}
\begin{proof}
Let $(\textbf{d}, \textbf{r})$ be a smooth arithmetical structure on $\CT{p}{s}$ with $d_p > 1$. If $p=1$, then \[\sum_{j=1}^s r_{\ell_j} =d_pr_p> r_p\ \mathrm{because}\  d_p > 1.\]
If $p > 1$, then \[d_pr_p = r_{p-1} + \sum_{j=1}^s r_{\ell_j}.\] Since the arithmetical structure is smooth, by Lemma~\ref{lem:CT-increasing-r}, 
we have $r_p > r_{p-1}$. Thus, \begin{equation*}
    d_pr_p = r_{p-1} + \sum_{j=1}^s r_{\ell_j} < r_p + \sum_{j=1}^s r_{\ell_j}.
\end{equation*} 
Subtracting $r_p$ from both sides and combining terms yields
\[\sum_{j=1}^s r_{\ell_j} > (d_p - 1)r_p>r_p.\] 

For the reverse direction, assume that $r_p < \sum_{j=1}^s r_{\ell_j}$.  Note that $p > 1$, because $r_p \neq \sum_{j=1}^s r_{\ell_j}$. 
Then
\[d_pr_p = r_{p-1} + \sum_{j=1}^s r_{\ell_j}>r_{p-1}+r_p>r_p,\]
which implies that $d_p>1.$
\end{proof}

Next, we give an alternative proof to Proposition 2.3 in Archer et al. \cite{archer_2020}, which stated that any smooth arithmetical structures on bidents, denoted as $\CT{p}{2}$, must have center vertex equal to~1.

\begin{corollary}
Let $p\geq 1$. There are no smooth arithmetical structures $(\mathbf{d}, \mathbf{r})$ on $\CT{p}{2}$ with $d_p > 1$.
\end{corollary}
\begin{proof}
Let $(\mathbf{d},\mathbf{r})$ be a smooth arithmetical structure on $\CT{p}{2}$. Then $r_{\ell_1}$ and $r_{\ell_2}$ are proper divisors of $r_p$. By Proposition~\ref{prop:CT-when-is-d_p>1}, $d_p>1$ if and only if $r_{\ell_1}+r_{\ell_2}>r_p$. So we need $r_{\ell_1}+r_{\ell_2}>r_p$. However, since $r_{\ell_1}$ and $r_{\ell_2}$ are proper divisors of $r_p$, we have $r_{\ell_1}+r_{\ell_2}\leq r_p$. 
Therefore, $d_p$ can not be greater~than~1.
\end{proof}

Next, given the numbers assigned to leaf vertices, we count the number of smooth arithmetical structures that have a label greater than one at the central vertex.

\begin{proposition}\label{prop:lcm}
Let $(r_{\ell_1}, \dots, r_{\ell_s}) \in \N^{s}$, let $S = \{n \in \N\mid n \cdot \lcm(r_{\ell_1}, \dots, r_{\ell_s}) < \sum_{j=1}^s r_{\ell_j}\}$.
 Then, given $(r_{\ell_1},\ldots,r_{\ell_s})$, the number of smooth arithmetical structures that can be constructed  such that the central vertex has label greater than 1 is is given by 
\begin{equation*}
|S|=
                       \left\lfloor \frac{\sum_{j=1}^s r_{\ell_j}}{\lcm(r_{\ell_1}, \dots, r_{\ell_s})} \right\rfloor.
\end{equation*}
\end{proposition}
\begin{proof}
First, consider when $|S| = 0$, which implies $\lcm(r_{\ell_1}, \dots, r_{\ell_s}) \geq \sum_{j=1}^s r_{\ell_j}$. If there is a structure $(\mathbf{d},\mathbf{r})\in\CT{p}{s}$ for some $p\in \N$, then since each $r_{\ell_j}$ divides $r_p$, we have $\lcm(r_{\ell_1}, \dots, r_{\ell_s}) \mid r_p$. Thus, $r_p \geq \lcm(r_{\ell_1}, \dots, r_{\ell_s}) \geq \sum_{j=1}^s r_{\ell_j}$. This contradicts Proposition~\ref{prop:CT-when-is-d_p>1}, hence there exists no smooth arithmetical structures that can be constructed such that $d_p > 1$.

Now, consider when $|S| \geq 1$. Let $r_{p_i}=i\cdot \lcm (r_{\ell_1}, \dots, r_{\ell_s})$ for $i=1,2,\ldots, |S|.$ By Proposition \ref{prop:CT-construction-as-unique}, we can construct $|S|$ smooth arithmetical structures, one per each value of $r_{p_i}$ assigned to the central vertex. By Proposition~\ref{prop:CT-when-is-d_p>1}, each of these structures have $d_{p_i}>1$. Moreover, since in any structure of the form $(r_1,r_2,\ldots, r_p,r_{\ell_1},r_{\ell_2},\ldots, r_{\ell_s})$ we must have $r_p\mid\lcm(r_{\ell_1}, \dots, r_{\ell_s})$ and the only values of $r_p$ that satisfy that $r_p<\sum_{j=1}^s r_{\ell_j}$ are $r_{p_i}$ for $i=1,2,\ldots, |S|$, then by Proposition~\ref{prop:CT-when-is-d_p>1} these are the only possible smooth structures with $d_p>1$.
To complete the proof, note that by the definition of $S$, $|S|$ is the largest integer $m$ such that $m < \frac{\sum_{j=1}^s r_{\ell_j}}{\lcm(r_{\ell_1}, r_{\ell_s})}$. Hence, $|S| = \left\lfloor \frac{\sum_{j=1}^s r_{\ell_j}}{\lcm(r_{\ell_1}, \dots, r_{\ell_s})} \right\rfloor$.
\end{proof}

We conclude with an application of Proposition~\ref{prop:lcm}.
\begin{example}
Recalling Example~\ref{ex:CT-non-unique-AS}, note that $\lcm(r_{\ell_1}, \dots, r_{\ell_5}) = 6$ and $\sum_{j=1}^5 r_{\ell_j} = 13$. 
Thus, the number of smooth arithmetical structures with label at the center vertex greater than 1 is  $\lfloor \frac{13}{6} \rfloor = 2$, which corresponds to when $r_{p_1} = 6$ and $r_{p_2} = 12$. These are precisely the two structures presented in Example~\ref{ex:CT-non-unique-AS}, namely, 
\[ \textbf{r}=(1,2,3,4,5,6,2,2,3,3,3) \text { and } \textbf{r}=(1,2,3,4,5,6,7,8,9,10,11,12,2,2,3,3,3).\]
\end{example}

\section{Future Work}\label{sec:futurework}
We conclude with some direction for future research. To begin we recall that Archer et al.~\cite{archer_2020} establish that the number of smooth arithmetical structures on a bident is bounded by a cubic polynomial. We ask the following:
\begin{question}
    Does there exist a polynomial bound for the number of smooth arithmetical structures on coconut trees $\CT{p}{s}$ for $s>2$?
\end{question}
\begin{question}
    If a polynomial bound for the number of smooth arithmetical structures on coconut trees $\CT{p}{s}$ for $s>2$  exists, how does it relate to $s$ the number of leafs?
\end{question}
We now pose a number theoretic question which would help in enumerating smooth arithmetical structures on $\CT{p}{s}$ if we are given $r_p$ and the sum of the leaves.
\begin{question}\label{q:number theory}
    Given integers $r_p$ and $\sum_{j=1}^s r_{\ell_j}$, in how many ways can we partition 
    $\sum_{j=1}^s r_{\ell_j}$ using the proper divisors of $r_p$?
\end{question}

\bibliographystyle{plain}
\bibliography{Bibliography.bib}

\end{document}